\documentclass[11pt]{amsart}

\usepackage[margin=1.25in]{geometry}
\usepackage{graphicx, hyperref, tikz, enumerate, amssymb, comment}
\usetikzlibrary{matrix,arrows,decorations.pathmorphing}

\newtheorem{theorem}{Theorem}[section]

\newtheorem{proposition}[theorem]{Proposition}
\newtheorem{corollary}[theorem]{Corollary}

\theoremstyle{definition}
\newtheorem{definition}[theorem]{Definition}
\newtheorem{example}[theorem]{Example}

\theoremstyle{remark}
\newtheorem{remark}[theorem]{Remark}

\numberwithin{equation}{section}

\DeclareMathOperator{\Ker}{Ker}

\DeclareMathOperator{\Ig}{Im}
\DeclareMathOperator{\Hg}{H}

\DeclareMathOperator{\id}{id}

\DeclareMathOperator{\T}{\mathbf{Trip}}

\begin{document}

\allowdisplaybreaks

\title{Properties of the Secondary Hochschild Homology}

\author{Jacob Laubacher}
\address{Department of Mathematics and Statistics, Bowling Green State University, Bowling Green, Ohio 43403}
\email{jlaubac@bgsu.edu}

\subjclass[2010]{Primary 16E40; Secondary 16D90}

\date{\today}

\keywords{Hochschild homology, Morita equivalence}

\begin{abstract}
In this paper we study properties of the secondary Hochschild homology of the triple $(A,B,\varepsilon)$ with coefficients in $M$. We establish a type of Morita equivalence between two triples and show that $\Hg_\bullet((A,B,\varepsilon);M)$ is invariant under this equivalence. We also prove the existence of an exact sequence which connects the usual and the secondary Hochschild homologies in low dimension, allowing one to perform easy computations. The functoriality of $\Hg_\bullet((A,B,\varepsilon);M)$ is also discussed.
\end{abstract}

\maketitle

\section*{Introduction}

Hochschild cohomology was introduced by Hochschild in \cite{H} as a method to study extensions of an associative algebra $A$ over a field $k$. Later Gerstenhaber exploited this to study deformations in \cite{G}. It's dual, the Hochschild homology, is used as both a stepping stone towards cyclic homology and a generalization of the modules of differential forms for noncommutative $k$-algebras $A$. The groups $\Hg_\bullet(A,M)$ (where $M$ is an $A$-bimodule) are Morita invariant.

Secondary Hochschild homology was introduced in \cite{LSS} through the use of simplicial algebras and simplicial modules. The main ingredient was the bar simplicial module $\mathcal{B}(A,B,\varepsilon)$ which behaves similar to the bar resolution associated to an algebra. The groups $\Hg_\bullet((A,B,\varepsilon);M)$ involve a triple $(A,B,\varepsilon)$ which consists of a commutative $k$-algebra $B$ inducing a $B$-algebra structure on $A$ by way of a morphism $\varepsilon:B\longrightarrow A$. Just as in the usual Hochschild homology, $M$ is taken to be an $A$-bimodule, but here we add the restriction that $M$ is also $B$-symmetric.  One goal of this paper is to show that the secondary Hochschild homology has a type of Morita invariance.

This paper is organized as follows: in the first section we recall the secondary Hochschild homology. We also review some basic results so as to keep this paper self-contained. In the second section we introduce the notion of Morita equivalence between two triples $(A,B,\varepsilon)$ and $(A',B',\varepsilon')$. Here we require two additional conditions to the usual definition of Morita equivalence between two $k$-algebras. With this in hand, we prove that the secondary Hochschild homology is Morita invariant (see Theorem \ref{preserve}). In particular, we show that $\Hg_\bullet((A,B,\varepsilon);M)\cong\Hg_\bullet((M_n(A),I_n(B),\varepsilon_*);M_n(M))$. In the final section we give some computations of the secondary Hochschild homology in low dimension. When $A$ is commutative we give the relation between $\Hg_1((A,B,\varepsilon);M)$ and K\"ahler differentials (see Proposition \ref{Kahler2}). We also introduce an exact sequence \eqref{sequence} which connects $\Hg_i((A,B,\varepsilon);M)$, $\Hg_i(A,M)$, and $\Hg_1(B,M)$ (for $i=1,2$). We conclude with a discussion about functoriality.

\section{Preliminaries}

In this paper we fix $k$ to be a field. We let all tensor products be over $k$ unless otherwise stated (that is, $\otimes=\otimes_k$). Furthermore, all $k$-algebras have multiplicative unit.

Fix $A$ to be an associative $k$-algebra, $B$ a commutative $k$-algebra, and $\varepsilon:B\longrightarrow A$ a morphism of $k$-algebras such that $\varepsilon(B)\subseteq\mathcal{Z}(A)$. By referring to a triple $(A,B,\varepsilon)$, we are invoking the above conditions. To say that a triple $(A,B,\varepsilon)$ is commutative corresponds to taking $A$ commutative. Finally, we let $M$ be an $A$-bimodule which is $B$-symmetric (that is, $m\varepsilon(\alpha)=\varepsilon(\alpha)m$ for all $m\in M$ and $\alpha\in B$).

\subsection{The Hochschild homology}

Recall from \cite{H}, \cite{L}, or \cite{W} the Hochschild homology. Define $C_n(A,M)=M\otimes A^{\otimes n}$ and $d_n:C_n(A,M)\longrightarrow C_{n-1}(A,M)$ determined by
\begin{equation*}
\begin{gathered}
d_n(m\otimes a_1\otimes\cdots\otimes a_n)=ma_1\otimes a_2\otimes\cdots\otimes a_n\\
+\sum_{i=1}^{n-1}(-1)^im\otimes a_1\otimes\cdots\otimes a_ia_{i+1}\otimes\cdots\otimes a_n+(-1)^na_nm\otimes a_1\otimes\cdots\otimes a_{n-1},
\end{gathered}
\end{equation*}
where $m\in M$ and $a_i\in A$. One can show that $d_nd_{n+1}=0$. We denote the chain complex
$$
\ldots\xrightarrow{d_{n+1}}M\otimes A^{\otimes n}\xrightarrow{~d_n~}M\otimes A^{\otimes n-1}\xrightarrow{d_{n-1}}\ldots\xrightarrow{~d_3~}M\otimes A^{\otimes2}\xrightarrow{~d_2~}M\otimes A\xrightarrow{~d_1~}M\longrightarrow0
$$
by $\mathbf{C}_\bullet(A,M)$.

\begin{definition}(\cite{H})
The homology of the complex $\mathbf{C}_\bullet(A,M)$ is called the \textbf{Hochschild homology of $A$ with coefficients in $M$} and is denoted by $\Hg_\bullet(A,M)$.
\end{definition}

Of particular interest is the case when one takes $M=A$ where $A$ is commutative. As seen in most homological algebra texts (such as \cite{L} or \cite{W}), one can connect the Hochschild homology with K\"ahler differentials.

\begin{proposition}\label{HH1A}\emph{(\cite{L},\cite{W})}
For a commutative $k$-algebra $A$ and an $A$-symmetric $A$-bimodule $M$, we have that
$$\Hg_1(A,M)\cong M\otimes_A\Omega_{A|k}^1,$$
and in particular $\Hg_1(A,A)\cong\Omega_{A|k}^1$.
\end{proposition}

\begin{theorem}\emph{(\cite{L},\cite{Mc})}
If $(P,Q)$ gives a Morita equivalence of $k$-algebras between $A$ and $A'$, then there is a natural isomorphism
$$\Hg_\bullet(A,M)\cong\Hg_\bullet(A',Q\otimes_AM\otimes_AP).$$
\end{theorem}

\subsection{The secondary Hochschild homology}

Recall the secondary Hochschild homology from \cite{LSS}. Define $C_n((A,B,\varepsilon);M)=M\otimes A^{\otimes n}\otimes B^{\otimes\frac{n(n-1)}{2}}$ and $\partial_n^{\varepsilon}:C_n((A,B,\varepsilon);M)\longrightarrow C_{n-1}((A,B,\varepsilon);M)$ determined by
$$
\partial_n^{\varepsilon}\left(m
\otimes
\begin{pmatrix}
a_1 & b_{1,2} & b_{1,3} & \cdots & b_{1,n-2} & b_{1,n-1} & b_{1,n}\\
1 & a_2 & b_{2,3} & \cdots & b_{2,n-2} & b_{2,n-1} & b_{2,n}\\
1 & 1 & a_3 & \cdots & b_{3,n-2} & b_{3,n-1} & b_{3,n}\\
\vdots & \vdots & \vdots & \ddots & \vdots & \vdots & \vdots\\
1 & 1 & 1 & \cdots & a_{n-2} & b_{n-2,n-1} & b_{n-2,n}\\
1 & 1 & 1 & \cdots & 1 & a_{n-1} & b_{n-1,n}\\
1 & 1 & 1 & \cdots & 1 & 1 & a_n\\
\end{pmatrix}
\right)
$$
$$
=ma_1\varepsilon(b_{1,2}b_{1,3}\cdots b_{1,n-2}b_{1,n-1}b_{1,n})
\otimes
\begin{pmatrix}
a_2 & b_{2,3} & \cdots & b_{2,n-2} & b_{2,n-1} & b_{2,n}\\
1 & a_3 & \cdots & b_{3,n-2} & b_{3,n-1} & b_{3,n}\\
\vdots & \vdots & \ddots & \vdots & \vdots & \vdots\\
1 & 1 & \cdots & a_{n-2} & b_{n-2,n-1} & b_{n-2,n}\\
1 & 1 & \cdots & 1 & a_{n-1} & b_{n-1,n}\\
1 & 1 & \cdots & 1 & 1 & a_n\\
\end{pmatrix}
$$
$$
+\sum_{i=1}^{n-1}(-1)^im
\otimes
\begin{pmatrix}
a_1 & b_{1,2} & \cdots & b_{1,i}b_{1,i+1} & \cdots & b_{1,n-1} & b_{1,n}\\
1 & a_2 & \cdots & b_{2,i}b_{2,i+1} & \cdots & b_{2,n-1} & b_{2,n}\\
\vdots & \vdots & \ddots & \vdots & \ddots & \vdots & \vdots\\
1 & 1 & \cdots & a_i\varepsilon(b_{i,i+1})a_{i+1} & \cdots & b_{i,n-1}b_{i+1,n-1} & b_{i,n}b_{i+1,n}\\
\vdots & \vdots & \ddots & \vdots & \ddots & \vdots & \vdots\\
1 & 1 & \cdots & 1 & \cdots & a_{n-1} & b_{n-1,n}\\
1 & 1 & \cdots & 1 & \cdots & 1 & a_n\\
\end{pmatrix}
$$
$$
+(-1)^na_n\varepsilon(b_{n-1,n}b_{n-2,n}\cdots b_{3,n}b_{2,n}b_{1,n})m
\otimes
\begin{pmatrix}
a_1 & b_{1,2} & b_{1,3} & \cdots & b_{1,n-2} & b_{1,n-1}\\
1 & a_2 & b_{2,3} & \cdots & b_{2,n-2} & b_{2,n-1}\\
1 & 1 & a_3 & \cdots & b_{3,n-2} & b_{3,n-1}\\
\vdots & \vdots & \vdots & \ddots & \vdots & \vdots\\
1 & 1 & 1 & \cdots & a_{n-2} & b_{n-2,n-1}\\
1 & 1 & 1 & \cdots & 1 & a_{n-1}\\
\end{pmatrix},
$$
where $m\in M$, $a_i\in A$, and $b_{i,j}\in B$. It was shown in \cite{LSS} that $\partial_n^\varepsilon\partial_{n+1}^\varepsilon=0$. We denote the chain complex
$$
\ldots\xrightarrow{\partial_{n+1}^\varepsilon}M\otimes A^{\otimes n}\otimes B^{\otimes\frac{n(n-1)}{2}}\xrightarrow{~\partial_n^\varepsilon~}M\otimes A^{\otimes n-1}\otimes B^{\otimes\frac{(n-1)(n-2)}{2}}\xrightarrow{\partial_{n-1}^\varepsilon}\ldots
$$
$$
\ldots\xrightarrow{~\partial_5^\varepsilon~}M\otimes A^{\otimes4}\otimes B^{\otimes6}\xrightarrow{~\partial_4^\varepsilon~}M\otimes A^{\otimes3}\otimes B^{\otimes3}\xrightarrow{~\partial_3^\varepsilon~}M\otimes A^{\otimes2}\otimes B\xrightarrow{~\partial_2^\varepsilon~}M\otimes A\xrightarrow{~\partial_1^\varepsilon~}M\longrightarrow0
$$
by $\mathbf{C}_\bullet((A,B,\varepsilon);M)$.

\begin{definition}\label{SHH}(\cite{LSS})
The homology of the complex $\mathbf{C}_\bullet((A,B,\varepsilon);M)$ is called the \textbf{secondary Hochschild homology of the triple $(A,B,\varepsilon)$ with coefficients in $M$} and is denoted by $\Hg_\bullet((A,B,\varepsilon);M)$.
\end{definition}

\begin{example}(\cite{LSS})
Notice that when $B=k$, we get the usual Hochschild homology. That is, $\Hg_\bullet((A,k,\varepsilon);M)=\Hg_\bullet(A,M)$, and so $\Hg_n((A,k,\varepsilon);M)\cong\Hg_n(A,M)$ for all $n\geq0$.
\end{example}

\begin{example}(\cite{LSS})
Observe that $\Hg_0((A,B,\varepsilon);M)=\Hg_0(A,M)=\frac{M}{[M,A]}$.
\end{example}

\section{Morita Equivalence of Triples}

The classical result of the usual Hochschild homology preserving Morita equivalence is well-known (see \cite{J}, \cite{L}, \cite{Mc}, or \cite{W}). In this section we establish the theory behind two triples being Morita equivalent and produce a similar result. Recall that $M$ is an $A$-bimodule which is $B$-symmetric.

\begin{definition}\label{MET}
Let $(A,B,\varepsilon)$ and $(A',B',\varepsilon')$ be two triples. We say that $(A,B,\varepsilon)$ and $(A',B',\varepsilon')$ are \textbf{Morita equivalent as triples} if
\begin{enumerate}[(i)]
\item there exists an $A-A'$-bimodule $P$ and an $A'-A$-bimodule $Q$ such that there is an isomorphism of $A$-bimodules $f:P\otimes_{A'}Q\longrightarrow A$ as well as an isomorphism of $A'$-bimodules $g:Q\otimes_AP\longrightarrow A'$,
\item there is an isomorphism of $k$-algebras $\eta:B\longrightarrow B'$, and
\item both $P$ and $Q$ are symmetric with respect to $B$ and $B'$ under $\eta$. That is,
$$\varepsilon(\alpha)p=p\varepsilon'\big(\eta(\alpha)\big)\hspace{.15in}\text{and}\hspace{.15in}q\varepsilon(\alpha)=\varepsilon'\big(\eta(\alpha)\big)q$$
for all $p\in P$, $q\in Q$, and $\alpha\in B$.
\end{enumerate}
\end{definition}

\begin{remark}
Condition (i) above says that $A$ and $A'$ are Morita equivalent as $k$-algebras. Condition (ii) says the same thing for $B$ and $B'$ since both are commutative.
\end{remark}

\begin{remark}
When $B$ and $B'$ are both equal to $k$, Definition \ref{MET} reduces to the usual definition of Morita equivalence of $k$-algebras between $A$ and $A'$.
\end{remark}

\begin{example}
Consider the triple $(A,B,\varepsilon)$. Let $e$ be an idempotent in $A$ such that $A=AeA$. Then $(A,B,\varepsilon)$ and $(eAe,B,\varepsilon_e)$ are Morita equivalent as triples where $\varepsilon_e:B\longrightarrow eAe$ is given by $\varepsilon_e(\alpha)=e\varepsilon(\alpha)e$ for all $\alpha\in B$.

It is easy to verify that $(eAe,B,\varepsilon_e)$ is a triple, and one can check the equivalence by setting $P:=Ae$, $Q:=eA$, and $\eta:=\id_B$.
\end{example}

\begin{proposition}
Morita equivalence of triples defines an equivalence relation.
\end{proposition}
\begin{proof}
Morita equivalence of triples is clearly both reflexive and symmetric. We need only show that it is transitive.

Suppose that $(P_1,Q_1,\eta_1)$ gives a Morita equivalence of triples between $(A,B,\varepsilon)$ and $(A',B',\varepsilon')$, and that $(P_2,Q_2,\eta_2)$ gives a Morita equivalence of triples between $(A',B',\varepsilon')$ and $(A'',B'',\varepsilon'')$. We will show that $(A,B,\varepsilon)$ and $(A'',B'',\varepsilon'')$ are Morita equivalent as triples.

Setting $P:=P_1\otimes_{A'}P_2$ and $Q:=Q_2\otimes_{A'}Q_1$, we get the isomorphisms $P\otimes_{A''}Q\cong A$ and $Q\otimes_AP\cong A''$. Thus, (i) is satisfied. For (ii), $\eta:B\longrightarrow B''$ is defined by the composition $\eta:=\eta_2\circ\eta_1$, which is still an isomorphism. Finally for (iii) we have that
\begin{align*}
\varepsilon(\alpha)p&=\varepsilon(\alpha)(p_1\otimes_{A'}p_2)\\
&=\varepsilon(\alpha)p_1\otimes_{A'}p_2\\
&=p_1\varepsilon'\big(\eta_1(\alpha)\big)\otimes_{A'}p_2\\
&=p_1\otimes_{A'}\varepsilon'\big(\eta_1(\alpha)\big)p_2\\
&=p_1\otimes_{A'}p_2\varepsilon''\big(\eta_2\circ\eta_1(\alpha)\big)\\
&=(p_1\otimes_{A'}p_2)\varepsilon''\big(\eta(\alpha)\big)\\
&=p\varepsilon''\big(\eta(\alpha)\big).
\end{align*}
Notice that $q\varepsilon(\alpha)=\varepsilon''\big(\eta(\alpha)\big)q$ in a similar way. Thus, transitivity follows and we have that Morita equivalence of triples defines an equivalence relation.
\end{proof}

\begin{remark}\label{Bimod}
Suppose $(P,Q,\eta)$ gives a Morita equivalence of triples between $(A,B,\varepsilon)$ and $(A',B',\varepsilon')$. Then $Q\otimes_AM\otimes_AP$ is clearly an $A'$-bimodule, and is also $B'$-symmetric since
\begin{align*}
\alpha'\cdot(q\otimes_Am\otimes_Ap)&=\varepsilon'(\alpha')q\otimes_Am\otimes_Ap\\
&=q\varepsilon(\alpha)\otimes_Am\otimes_Ap\\
&=q\otimes_A\varepsilon(\alpha)m\otimes_Ap\\
&=q\otimes_Am\varepsilon(\alpha)\otimes_Ap\\
&=q\otimes_Am\otimes_A\varepsilon(\alpha)p\\
&=q\otimes_Am\otimes_Ap\varepsilon'(\alpha')\\
&=(q\otimes_Am\otimes_Ap)\cdot\alpha',
\end{align*}
where $\eta^{-1}(\alpha')=\alpha$ and thus $\eta(\alpha)=\alpha'$.
\end{remark}

\begin{theorem}\label{preserve}
If $(P,Q,\eta)$ gives a Morita equivalence of triples between $(A,B,\varepsilon)$ and $(A',B',\varepsilon')$, then there is a natural isomorphism
$$\Hg_\bullet((A,B,\varepsilon);M)\cong\Hg_\bullet((A',B',\varepsilon');Q\otimes_AM\otimes_AP).$$
\end{theorem}
\begin{proof}
For ease of notation, throughout this proof we denote $\ominus:=\otimes_A$ and $\odot:=\otimes_{A'}$ where appropriate. We will follow the line of proof from \cite{L} and recall that $f:P\odot Q\longrightarrow A$ and $g:Q\ominus P\longrightarrow A'$ are bimodule isomorphisms. Observe that $f$ and $g$ satisfy
\begin{equation}\label{Lod}
q_1f(p_1\odot q_2)=g(q_1\ominus p_1)q_2\hspace{.15in}\text{and}\hspace{.15in}p_1g(q_1\ominus p_2)=f(q_1\odot q_1)p_2
\end{equation}
for all $p_1,p_2\in P$ and $q_1,q_2\in Q$. One can then view $f$ and $g$ as ring homomorphisms with the product defined as follows:
$$(p_1\odot q_1)(p_2\odot q_2)=p_1\odot g(q_1\ominus p_2)q_2\hspace{.15in}\text{and}\hspace{.15in}(q_1\ominus p_1)(q_2\ominus p_2)=q_1\ominus f(p_1\odot q_2)p_2.$$

Next, because $f$ and $g$ are isomorphisms, there exists $p_1,\ldots,p_s\in P$ and $q_1,\ldots,q_s\in Q$, as well as $p_1',\ldots,p_t'\in P$ and $q_1',\ldots,q_t'\in Q$, such that
$$f\left(\sum_{j=1}^sp_j\odot q_j\right)=1_A\hspace{.15in}\text{and}\hspace{.15in}g\left(\sum_{m=1}^tq_m'\ominus p_m'\right)=1_{A'}.$$

For every $n\geq0$ define $\psi_n:M\otimes A^{\otimes n}\otimes B^{\otimes\frac{n(n-1)}{2}}\longrightarrow(Q\otimes_AM\otimes_AP)\otimes {A'}^{\otimes n}\otimes{B'}^{\otimes\frac{n(n-1)}{2}}$ by
$$
\psi_n\left(m
\otimes
\begin{pmatrix}
a_1 & b_{1,2} & \cdots & b_{1,n-1} & b_{1,n}\\
1 & a_2 & \cdots & b_{2,n-1} & b_{2,n}\\
\vdots & \vdots & \ddots & \vdots & \vdots\\
1 & 1 & \cdots & a_{n-1} & b_{n-1,n}\\
1 & 1 & \cdots & 1 & a_n\\
\end{pmatrix}
\right)=\sum q_{j_0}\otimes_Am\otimes_Ap_{j_1}
$$
$$
\otimes
\begin{pmatrix}
g(q_{j_1}\ominus a_1p_{j_2}) & \eta(b_{1,2}) & \cdots & \eta(b_{1,n-1}) & \eta(b_{1,n})\\
1 & g(q_{j_2}\ominus a_2p_{j_3}) & \cdots & \eta(b_{2,n-1}) & \eta(b_{2,n})\\
\vdots & \vdots & \ddots & \vdots & \vdots\\
1 & 1 & \cdots & g(q_{j_{n-1}}\ominus a_{n-1}p_{j_n}) & \eta(b_{n-1,n})\\
1 & 1 & \cdots & 1 & g(q_{j_n}\ominus a_np_{j_0})\\
\end{pmatrix},
$$
where the sum is taken over all sets of indices $(j_0,j_1,\ldots,j_n)$ such that $1\leq j_i\leq s$ for $0\leq i\leq n$. Furthermore define $\varphi_n:(Q\otimes_AM\otimes_AP)\otimes {A'}^{\otimes n}\otimes{B'}^{\otimes\frac{n(n-1)}{2}}\longrightarrow M\otimes A^{\otimes n}\otimes B^{\otimes\frac{n(n-1)}{2}}$ determined by
$$
\varphi_n\left(q\otimes_Am\otimes_Ap
\otimes
\begin{pmatrix}
a_1' & b_{1,2}' & \cdots & b_{1,n-1}' & b_{1,n}'\\
1 & a_2' & \cdots & b_{2,n-1}' & b_{2,n}'\\
\vdots & \vdots & \ddots & \vdots & \vdots\\
1 & 1 & \cdots & a_{n-1}' & b_{n-1,n}'\\
1 & 1 & \cdots & 1 & a_n'\\
\end{pmatrix}
\right)=\sum f(p_{m_0}'\odot q)mf(p\odot q_{m_1}')
$$
$$
\otimes
\begin{pmatrix}
f(p_{m_1}'\odot a_1'q_{m_2}') & \eta^{-1}(b_{1,2}') & \cdots & \eta^{-1}(b_{1,n-1}') & \eta^{-1}(b_{1,n}')\\
1 & f(p_{m_2}'\odot a_2'q_{m_3}') & \cdots & \eta^{-1}(b_{2,n-1}') & \eta^{-1}(b_{2,n}')\\
\vdots & \vdots & \ddots & \vdots & \vdots\\
1 & 1 & \cdots & f(p_{m_{n-1}}'\odot a_{n-1}'q_{m_n}') & \eta^{-1}(b_{n-1,n}')\\
1 & 1 & \cdots & 1 & f(p_{m_n}'\odot a_n'q_{m_0}')\\
\end{pmatrix},
$$
where the sum is taken over all sets of indices $(m_0,m_1,\ldots,m_n)$ such that $1\leq m_i\leq t$ for $0\leq i\leq n$. Both $\psi$ and $\varphi$ are morphisms of complexes due to \eqref{Lod}.

There is a presimplicial homotopy $h$ between the composite $\varphi\circ\psi$ and $\id_{\mathbf{C}_\bullet((A,B,\varepsilon);M)}$ given by
$$
h_i\left(m
\otimes
\begin{pmatrix}
a_1 & b_{1,2} & \cdots & b_{1,n-1} & b_{1,n}\\
1 & a_2 & \cdots & b_{2,n-1} & b_{2,n}\\
\vdots & \vdots & \ddots & \vdots & \vdots\\
1 & 1 & \cdots & a_{n-1} & b_{n-1,n}\\
1 & 1 & \cdots & 1 & a_n\\
\end{pmatrix}
\right)=\sum mf(p_{j_0}\odot q_{m_0}')
$$
$$
\resizebox{\linewidth}{!}{$
\otimes\begin{pmatrix}
f(p_{m_0}'\odot q_{j_0})a_1f(p_{j_1}\odot q_{m_1}') & \cdots & b_{1,i} & 1 & b_{1,i+1} & \cdots & b_{1,n}\\
\vdots & \ddots & \vdots & \vdots & \vdots & \ddots & \vdots\\
1 & \cdots & f(p_{m_{i-1}}'\odot q_{j_{i-1}})a_if(p_{j_i}\odot q_{m_i}') & 1 & b_{i,i+1} & \cdots & b_{i,n}\\
1 & \cdots & 1 & f(p_{m_i}'\odot q_{j_i}) & 1 & \cdots & 1\\
1 & \cdots & 1 & 1 & a_{i+1} & \cdots & b_{i+1,n}\\
\vdots & \ddots & \vdots & \vdots & \vdots & \ddots & \vdots\\
1 & \cdots & 1 & 1 & 1 & \cdots & a_n\\
\end{pmatrix}
$}
$$
where the sum is taken over all sets of indices $(j_0,\ldots,j_i)$ and $(m_0,\ldots,m_i)$ such that $1\leq j_*\leq s$ and $1\leq m_*\leq t$. Likewise, there is a presimplicial homotopy $l$ between $\psi\circ\varphi$ and $\id_{\mathbf{C}_\bullet((A',B',\varepsilon');Q\otimes_AM\otimes_AP)}$ given by
$$
l_i\left(q\otimes_Am\otimes_Ap
\otimes
\begin{pmatrix}
a_1' & b_{1,2}' & \cdots & b_{1,n-1}' & b_{1,n}'\\
1 & a_2' & \cdots & b_{2,n-1}' & b_{2,n}'\\
\vdots & \vdots & \ddots & \vdots & \vdots\\
1 & 1 & \cdots & a_{n-1}' & b_{n-1,n}'\\
1 & 1 & \cdots & 1 & a_n'\\
\end{pmatrix}
\right)=\sum q\otimes_Am\otimes_Apg(q_{m_0}'\ominus p_{j_0})
$$
$$
\resizebox{\linewidth}{!}{$
\otimes\begin{pmatrix}
g(q_{j_0}\ominus p_{m_0}')a_1'g(q_{m_1}'\ominus p_{j_1}) & \cdots & b_{1,i}' & 1 & b_{1,i+1}' & \cdots & b_{1,n}'\\
\vdots & \ddots & \vdots & \vdots & \vdots & \ddots & \vdots\\
1 & \cdots & g(q_{j_{i-1}}\ominus p_{m_{i-1}}')a_i'g(q_{m_i}'\ominus p_{j_i}) & 1 & b_{i,i+1}' & \cdots & b_{i,n}'\\
1 & \cdots & 1 & g(q_{j_i}\ominus p_{m_i}') & 1 & \cdots & 1\\
1 & \cdots & 1 & 1 & a_{i+1}' & \cdots & b_{i+1,n}'\\
\vdots & \ddots & \vdots & \vdots & \vdots & \ddots & \vdots\\
1 & \cdots & 1 & 1 & 1 & \cdots & a_n'\\
\end{pmatrix}
$}
$$
where the sum is taken over all sets of indices $(j_0,\ldots,j_i)$ and $(m_0,\ldots,m_i)$ such that $1\leq j_*\leq s$ and $1\leq m_*\leq t$.

One can verify that both the $h_i$'s and $l_i$'s form a presimplicial homotopy. Thus, $\varphi\circ\psi$ is homotopic to the identity on the complex $\mathbf{C}_\bullet((A,B,\varepsilon);M)$, and $\psi\circ\varphi$ is homotopic to the identity on the complex $\mathbf{C}_\bullet((A',B',\varepsilon');Q\otimes_AM\otimes_AP)$.

Hence, our desired isomorphism at the level of homology follows.
\end{proof}

Consider a triple $(A,B,\varepsilon)$. Define
$$I_n(B):=B\cdot I_n=\left\{\begin{pmatrix} \alpha & \cdots & 0\\ \vdots & \ddots & \vdots\\ 0 & \cdots & \alpha\\ \end{pmatrix}~:~\alpha\in B\right\}.$$
Notice that $M_n(A)$ is an associative $k$-algebra and $I_n(B)$ is a commutative $k$-algebra, both with multiplicative unit. Furthermore, $\varepsilon:B\longrightarrow A$ induces the map $\varepsilon_*:I_n(B)\longrightarrow M_n(A)$ given by
$$\varepsilon_*\left(\begin{pmatrix} \alpha & \cdots & 0\\ \vdots & \ddots & \vdots\\ 0 & \cdots & \alpha\\ \end{pmatrix}\right)=\begin{pmatrix} \varepsilon(\alpha) & \cdots & 0\\ \vdots & \ddots & \vdots\\ 0 & \cdots & \varepsilon(\alpha)\\ \end{pmatrix}.$$
Oberve $\varepsilon_*(I_n(B))\subseteq\mathcal{Z}(M_n(A))$ and hence $(M_n(A),I_n(B),\varepsilon_*)$ is a triple.

\begin{proposition}
We have that $(A,B,\varepsilon)$ and $(M_n(A),I_n(B),\varepsilon_*)$ are Morita equivalent as triples. In particular,
$$\Hg_\bullet((A,B,\varepsilon);M)\cong\Hg_\bullet((M_n(A),I_n(B),\varepsilon_*);M_n(M)).$$
\end{proposition}
\begin{proof}
Let $P$ be the module of row vectors $\begin{pmatrix} a_1 & a_2 & \cdots & a_n \end{pmatrix}$ of length $n$, and $Q$ be the module of column vectors $\begin{pmatrix} a_1 & a_2 & \cdots & a_n \end{pmatrix}^T$ of length $n$, both with entries from $A$. Note that $P$ is an $A-M_n(A)$-bimodule and $Q$ is an $M_n(A)-A$-bimodule with the actions of matrix multiplication. This yields natural bimodule isomorphisms $f:P\otimes_{M_n(A)}Q\longrightarrow A$ and $g:Q\otimes_AP\longrightarrow M_n(A)$. This is condition (i), which is the usual Morita equivalence between $A$ and $M_n(A)$. One can see \cite{L} or \cite{W} for more details.

Next, there is a natural isomorphism $\eta:B\longrightarrow I_n(B)$ given by
$$\eta(\alpha)=\begin{pmatrix} \alpha & \cdots & 0\\ \vdots & \ddots & \vdots\\ 0 & \cdots & \alpha\\ \end{pmatrix}$$
for all $\alpha\in B$. This establishes (ii).

For (iii) we have that
\begin{align*}
\varepsilon(\alpha)p&=\varepsilon(\alpha)\begin{pmatrix} a_1 & a_2 & \cdots & a_n \end{pmatrix}\\
&=\begin{pmatrix} \varepsilon(\alpha)a_1 & \varepsilon(\alpha)a_2 & \cdots & \varepsilon(\alpha)a_n \end{pmatrix}\\
&=\begin{pmatrix} a_1\varepsilon(\alpha) & a_2\varepsilon(\alpha) & \cdots & a_n\varepsilon(\alpha) \end{pmatrix}\\
&=\begin{pmatrix} a_1 & a_2 & \cdots & a_n \end{pmatrix}\begin{pmatrix} \varepsilon(\alpha) & 0 & \cdots & 0\\ 0 & \varepsilon(\alpha) & \cdots & 0\\ \vdots & \vdots & \ddots & \vdots\\ 0 & 0 & \cdots & \varepsilon(\alpha) \end{pmatrix}\\
&=\begin{pmatrix} a_1 & a_2 & \cdots & a_n \end{pmatrix}\varepsilon_*\left(\begin{pmatrix} \alpha & 0 & \cdots & 0\\ 0 & \alpha & \cdots & 0\\ \vdots & \vdots & \ddots & \vdots\\ 0 & 0 & \cdots & \alpha \end{pmatrix}\right)\\
&=\begin{pmatrix} a_1 & a_2 & \cdots & a_n \end{pmatrix}\varepsilon_*\big(\eta(\alpha)\big)\\
&=p\varepsilon_*\big(\eta(\alpha)\big).
\end{align*}
Observe $q\varepsilon(\alpha)=\varepsilon_*\big(\eta(\alpha)\big)q$ follows identically. Thus, $(A,B,\varepsilon)$ and $(M_n(A),I_n(B),\varepsilon_*)$ are Morita equivalent as triples.

For the isomorphism we invoke Theorem \ref{preserve} where $Q\otimes_AM\otimes_AP$ reduces to $M_n(M)$.
\end{proof}

\begin{remark}
One can also apply this concept of Morita equivalence of triples to the secondary Hochschild cohomology $\Hg^\bullet((A,B,\varepsilon);M)$, which was introduced in \cite{S} and studied more extensively in \cite{CSS}, \cite{LSS}, and \cite{SS}.
\end{remark}

\section{Computations and Functoriality}

Our goal in this section is to establish some computations of $\Hg_\bullet((A,B,\varepsilon);M)$ in low dimension, along with basic properties of its functoriality. The cohomology analogue of this section was done in \cite{SS}. First, recall the following maps used to define the secondary Hochschild homology.

\begin{remark}
We have that
$$\partial_1^{\varepsilon}(m\otimes a)=ma-am,$$
$$
\partial_2^{\varepsilon}\left(m\otimes
\begin{pmatrix}
a & \alpha\\
1 & b\\
\end{pmatrix}
\right)
=ma\varepsilon(\alpha)\otimes b-m\otimes a\varepsilon(\alpha)b+b\varepsilon(\alpha)m\otimes a,
$$
and
\begin{equation*}
\begin{gathered}
\partial_3^{\varepsilon}\left(m\otimes
\begin{pmatrix}
a & \alpha & \beta\\
1 & b & \gamma\\
1 & 1 & c\\
\end{pmatrix}
\right)
=ma\varepsilon(\alpha\beta)\otimes
\begin{pmatrix}
b & \gamma\\
1 & c\\
\end{pmatrix}
-m\otimes
\begin{pmatrix}
a\varepsilon(\alpha)b & \beta\gamma\\
1 & c\\
\end{pmatrix}\\
+m\otimes
\begin{pmatrix}
a & \alpha\beta\\
1 & b\varepsilon(\gamma)c\\
\end{pmatrix}
-c\varepsilon(\beta\gamma)m\otimes
\begin{pmatrix}
a & \alpha\\
1 & b\\
\end{pmatrix}.
\end{gathered}
\end{equation*}
\end{remark}

\subsection{Low-level computations}

We've seen that $\Hg_\bullet(A,M)$ relates to $k$-linear K\"ahler differentials (see Proposition \ref{HH1A}). It turns out that $\Hg_\bullet((A,B,\varepsilon);M)$ also corresponds to differentials, but in this case are $B$-linear.

\begin{proposition}\label{Kahler2}
For a commutative triple $(A,B,\varepsilon)$ and an $A$-symmetric $A$-bimodule $M$, we have that
$$\Hg_1((A,B,\varepsilon);M)\cong M\otimes_A\Omega_{A|B}^1,$$
and in particular $\Hg_1((A,B,\varepsilon);A)\cong\Omega_{A|B}^1$.
\end{proposition}
\begin{proof}
Since $M$ is $A$-symmetric, we get that the map $\partial_1^{\varepsilon}:M\otimes A\longrightarrow M$ is trivial. Therefore $\Hg_1((A,B,\varepsilon);M)$ is the quotient of $M\otimes A$ by the relation
\begin{equation}\label{brel}
ma\varepsilon(\alpha)\otimes b-m\otimes a\varepsilon(\alpha)b+b\varepsilon(\alpha)m\otimes a=0.
\end{equation}
The map $\Hg_1((A,B,\varepsilon);M)\longrightarrow M\otimes_A\Omega_{A|B}^1$ sends the class of $m\otimes a$ to $m\otimes_A d(a)$. Notice this is well-defined because \eqref{brel} maps to
$$ma\varepsilon(\alpha)\otimes_A d(b)-m\otimes_A d(a\varepsilon(\alpha) b)+b\varepsilon(\alpha)m\otimes_A d(a)=0$$
due to $B$-linearity.

Moreover, the map $M\otimes_A\Omega_{A|B}^1\longrightarrow\Hg_1((A,B,\varepsilon);M)$ sends $m\otimes_A ad(b)$ to the class of $ma\otimes b$, which is a cycle because $A$ is commutative and $M$ is $A$-symmetric. This is well-defined because $m\otimes_A d(ab)-m\otimes_A ad(b)-m\otimes_A bd(a)$ maps to
$$
m\otimes ab-ma\otimes b-mb\otimes a=0
$$
when we take $\alpha=1_B$ in \eqref{brel}.

Finally observe the two maps are inverses of each other, and the isomorphism follows.
\end{proof}

\begin{remark}
When $B=k$, Proposition \ref{Kahler2} reduces to Proposition \ref{HH1A}.
\end{remark}

\begin{example}
With $B=A$ (and in particular, $A$ is commutative and $\varepsilon=\id$), we have that $\Hg_1((A,A,\id);M)=0$ as consequence of Proposition \ref{Kahler2}.
\end{example}

\subsection{An exact sequence}

Next we show that the following sequence is exact for a triple $(A,B,\varepsilon)$:
\begin{equation}\label{sequence}
\begin{aligned}
\Hg_2(A,M)\xrightarrow{~\Phi^2~}\Hg_2((A,B,\varepsilon);M)\xrightarrow{~\Psi~}\Hg_1(B&,M)\xrightarrow{~\varepsilon_*~}\Hg_1(A,M)\\
&\xrightarrow{~\Phi^1~}\Hg_1((A,B,\varepsilon);M)\longrightarrow0.
\end{aligned}
\end{equation}
Define the above maps as follows:
$$\Phi^2(m\otimes a\otimes b)=m\otimes\begin{pmatrix} a & 1_B\\ 1 & b\\ \end{pmatrix},$$
$$\Psi\left(m\otimes\begin{pmatrix} a & \alpha\\ 1 & b\\ \end{pmatrix}\right)=bma\otimes\alpha,$$
$$\varepsilon_*(m\otimes\alpha)=m\otimes\varepsilon(\alpha),$$
and
$$\Phi^1(m\otimes a)=m\otimes a.$$
One can verify that these maps are well-defined.

\begin{proposition}\label{result}
Concerning the chain \eqref{sequence},
\begin{enumerate}[(i)]
\item $\Ig(\Phi^2)\subseteq\Ker(\Psi)$,
\item $\Ker(\Psi)\subseteq\Ig(\Phi^2)$,
\item $\Ig(\Psi)\subseteq\Ker(\varepsilon_*)$,
\item $\Ker(\varepsilon_*)\subseteq\Ig(\Psi)$,
\item $\Ig(\varepsilon_*)\subseteq\Ker(\Phi^1)$,
\item $\Ker(\Phi^1)\subseteq\Ig(\varepsilon_*)$, and
\item $\Phi^1$ is surjective.
\end{enumerate}
In particular,
\begin{align*}
\Hg_2(A,M)\xrightarrow{~\Phi^2~}\Hg_2((A,B,\varepsilon);M)\xrightarrow{~\Psi~}\Hg_1(B&,M)\xrightarrow{~\varepsilon_*~}\Hg_1(A,M)\\
&\xrightarrow{~\Phi^1~}\Hg_1((A,B,\varepsilon);M)\longrightarrow0
\end{align*}
is exact.
\end{proposition}
\begin{proof}
First observe that the class of elements of the form $m\otimes 1$ is zero in $\Hg_1(A,M)$, $\Hg_1(B,M)$, and $\Hg_1((A,B,\varepsilon);M)$. Parts $(i)$, $(v)$, and $(vii)$ are clear.

For $(ii)$, we take $m\otimes\begin{pmatrix} a & \alpha\\ 1 & b\\ \end{pmatrix}\in\Hg_2((A,B,\varepsilon);M)$ such that $bma\otimes\alpha=0$ in $\Hg_1(B,M)$ (that is, $m\otimes\begin{pmatrix} a & \alpha\\ 1 & b\\ \end{pmatrix}\in\Ker(\Psi)$). This means that our element is a boundary, and so there exists $n\in M$ and $\beta,\gamma\in B$ such that $d_2^B(n\otimes\beta\otimes\gamma)=bma\otimes\alpha$. Thus, we get that
\begin{align*}
bma\otimes\alpha=n\varepsilon(\beta)\otimes\gamma-n\otimes\beta\gamma+\varepsilon(\gamma)n\otimes\beta.
\end{align*}
Tensoring by $\otimes1_A\otimes1_A$ we now have
\begin{equation}\label{e1}
bma\otimes\begin{pmatrix} 1_A & \alpha\\ 1 & 1_A\\ \end{pmatrix}=n\varepsilon(\beta)\otimes\begin{pmatrix} 1_A & \gamma\\ 1 & 1_A\\ \end{pmatrix}-n\otimes\begin{pmatrix} 1_A & \beta\gamma\\ 1 & 1_A\\ \end{pmatrix}+\varepsilon(\gamma)n\otimes\begin{pmatrix} 1_A & \beta\\ 1 & 1_A\\ \end{pmatrix}.
\end{equation}
Further, we observe the following boundaries:
\begin{equation}\label{e3}
\begin{gathered}
\partial_3^{\varepsilon}\left(n\otimes\begin{pmatrix} 1_A & 1_B & \beta\\ 1 & 1_A & \gamma\\ 1 & 1 & 1_A\\ \end{pmatrix}\right)=n\varepsilon(\beta)\otimes\begin{pmatrix} 1_A & \gamma\\ 1 & 1_A\\ \end{pmatrix}-n\otimes\begin{pmatrix} 1_A & \beta\gamma\\ 1 & 1_A\\ \end{pmatrix}\\
+n\otimes\begin{pmatrix} 1_A & \beta\\ 1 & \varepsilon(\gamma)\\ \end{pmatrix}-\varepsilon(\beta)\varepsilon(\gamma)n\otimes\begin{pmatrix} 1_A & 1_B\\ 1 & 1_A\\ \end{pmatrix},
\end{gathered}
\end{equation}
and
\begin{equation}\label{e4}
\begin{gathered}
\partial_3^{\varepsilon}\left(n\otimes\begin{pmatrix} 1_A & \beta & 1_B\\ 1 & 1_A & 1_A\\ 1 & 1 & \varepsilon(\gamma)\\ \end{pmatrix}\right)=n\varepsilon(\beta)\otimes\begin{pmatrix} 1_A & 1_B\\ 1 & \varepsilon(\gamma)\\ \end{pmatrix}-n\otimes\begin{pmatrix} \varepsilon(\beta) & 1_B\\ 1 & \varepsilon(\gamma)\\ \end{pmatrix}\\
+n\otimes\begin{pmatrix} 1_A & \beta\\ 1 & \varepsilon(\gamma)\\ \end{pmatrix}-\varepsilon(\gamma)n\otimes\begin{pmatrix} 1_A & \beta\\ 1 & 1_A\\ \end{pmatrix}.
\end{gathered}
\end{equation}
Thus, we have that
\begin{align*}
bma\otimes&\begin{pmatrix} 1_A & \alpha\\ 1 & 1_A\\ \end{pmatrix}=n\varepsilon(\beta)\otimes\begin{pmatrix} 1_A & \gamma\\ 1 & 1_A\\ \end{pmatrix}-n\otimes\begin{pmatrix} 1_A & \beta\gamma\\ 1 & 1_A\\ \end{pmatrix}+\varepsilon(\gamma)n\otimes\begin{pmatrix} 1_A & \beta\\ 1 & 1_A\\ \end{pmatrix} \text{~~by \eqref{e1}}\\
&=\varepsilon(\beta)\varepsilon(\gamma)n\otimes\begin{pmatrix} 1_A & 1_B\\ 1 & 1_A\\ \end{pmatrix}-n\otimes\begin{pmatrix} 1_A & \beta\\ 1 & \varepsilon(\gamma)\\ \end{pmatrix}+\varepsilon(\gamma)n\otimes\begin{pmatrix} 1_A & \beta\\ 1 & 1_A\\ \end{pmatrix} \text{~~by \eqref{e3}}\\
&=\varepsilon(\beta)\varepsilon(\gamma)n\otimes\begin{pmatrix} 1_A & 1_B\\ 1 & 1_A\\ \end{pmatrix}-n\otimes\begin{pmatrix} \varepsilon(\beta) & 1_B\\ 1 & \varepsilon(\gamma)\\ \end{pmatrix}+n\varepsilon(\beta)\otimes\begin{pmatrix} 1_A & 1_B\\ 1 & \varepsilon(\gamma)\\ \end{pmatrix} \text{~~by \eqref{e4}}.
\end{align*}
We want to keep track of this, so formally observe from above, 
\begin{equation}\label{e5}
\begin{gathered}
bma\otimes\begin{pmatrix} 1_A & \alpha\\ 1 & 1_A\\ \end{pmatrix}\\
=\varepsilon(\beta)\varepsilon(\gamma)n\otimes\begin{pmatrix} 1_A & 1_B\\ 1 & 1_A\\ \end{pmatrix}-n\otimes\begin{pmatrix} \varepsilon(\beta) & 1_B\\ 1 & \varepsilon(\gamma)\\ \end{pmatrix}+n\varepsilon(\beta)\otimes\begin{pmatrix} 1_A & 1_B\\ 1 & \varepsilon(\gamma)\\ \end{pmatrix}.
\end{gathered}
\end{equation}
Next we will employ the two boundaries
\begin{equation}\label{e6}
\begin{gathered}
\partial_3^{\varepsilon}\left(m\otimes\begin{pmatrix} a & \alpha & 1_B\\ 1 & 1_A & 1_B\\ 1 & 1 & b\\ \end{pmatrix}\right)=ma\varepsilon(\alpha)\otimes\begin{pmatrix} 1_A & 1_B\\ 1 & b\\ \end{pmatrix}-m\otimes\begin{pmatrix} a\varepsilon(\alpha) & 1_B\\ 1 & b\\ \end{pmatrix}\\
+m\otimes\begin{pmatrix} a & \alpha\\ 1 & b\\ \end{pmatrix}-bm\otimes\begin{pmatrix} a & \alpha\\ 1 & 1_A\\ \end{pmatrix},
\end{gathered}
\end{equation}
and
\begin{equation}\label{e7}
\begin{gathered}
\partial_3^{\varepsilon}\left(bm\otimes\begin{pmatrix} a & 1_B & 1_B\\ 1 & 1_A & \alpha\\ 1 & 1 & 1_A\\ \end{pmatrix}\right)=bma\otimes\begin{pmatrix} 1_A & \alpha\\ 1 & 1_A\\ \end{pmatrix}-bm\otimes\begin{pmatrix} a & \alpha\\ 1 & 1_A\\ \end{pmatrix}\\
+bm\otimes\begin{pmatrix} a & 1_B\\ 1 & \varepsilon(\alpha)\\ \end{pmatrix}-\varepsilon(\alpha)bm\otimes\begin{pmatrix} a & 1_B\\ 1 & 1_A\\ \end{pmatrix}.
\end{gathered}
\end{equation}
So in $\Hg_2((A,B,\varepsilon);M)$, we have that
\begin{align*}
m\otimes\begin{pmatrix} a & \alpha\\ 1 & b\\ \end{pmatrix}&=bm\otimes\begin{pmatrix} a & \alpha\\ 1 & 1_A\\ \end{pmatrix}-ma\varepsilon(\alpha)\otimes\begin{pmatrix} 1_A & 1_B\\ 1 & b\\ \end{pmatrix}+m\otimes\begin{pmatrix} a\varepsilon(\alpha) & 1_B\\ 1 & b\\ \end{pmatrix} \text{~~by \eqref{e6}}\\
&=bma\otimes\begin{pmatrix} 1_A & \alpha\\ 1 & 1_A\\ \end{pmatrix}-\varepsilon(\alpha)bm\otimes\begin{pmatrix} a & 1_B\\ 1 & 1_A\\ \end{pmatrix}+bm\otimes\begin{pmatrix} a & 1_B\\ 1 & \varepsilon(\alpha)\\ \end{pmatrix}\\
&~-ma\varepsilon(\alpha)\otimes\begin{pmatrix} 1_A & 1_B\\ 1 & b\\ \end{pmatrix}+m\otimes\begin{pmatrix} a\varepsilon(\alpha) & 1_B\\ 1 & b\\ \end{pmatrix} \text{~~by \eqref{e7}}\\
&=\varepsilon(\beta)\varepsilon(\gamma)n\otimes\begin{pmatrix} 1_A & 1_B\\ 1 & 1_A\\ \end{pmatrix}-n\otimes\begin{pmatrix} \varepsilon(\beta) & 1_B\\ 1 & \varepsilon(\gamma)\\ \end{pmatrix}+n\varepsilon(\beta)\otimes\begin{pmatrix} 1_A & 1_B\\ 1 & \varepsilon(\gamma)\\ \end{pmatrix}\\
&~-\varepsilon(\alpha)bm\otimes\begin{pmatrix} a & 1_B\\ 1 & 1_A\\ \end{pmatrix}+bm\otimes\begin{pmatrix} a & 1_B\\ 1 & \varepsilon(\alpha)\\ \end{pmatrix}-ma\varepsilon(\alpha)\otimes\begin{pmatrix} 1_A & 1_B\\ 1 & b\\ \end{pmatrix}\\
&~+m\otimes\begin{pmatrix} a\varepsilon(\alpha) & 1_B\\ 1 & b\\ \end{pmatrix} \text{~~by \eqref{e5}}.
\end{align*}
Notice that we have expressed $m\otimes\begin{pmatrix} a & \alpha\\ 1 & b\\ \end{pmatrix}$ as a sum of seven elements with $1_B$ in the upper right of the matrix. So formally, we note
\begin{equation}\label{e8}
\begin{gathered}
m\otimes\begin{pmatrix} a & \alpha\\ 1 & b\\ \end{pmatrix}=\varepsilon(\beta)\varepsilon(\gamma)n\otimes\begin{pmatrix} 1_A & 1_B\\ 1 & 1_A\\ \end{pmatrix}-n\otimes\begin{pmatrix} \varepsilon(\beta) & 1_B\\ 1 & \varepsilon(\gamma)\\ \end{pmatrix}\\
+n\varepsilon(\beta)\otimes\begin{pmatrix} 1_A & 1_B\\ 1 & \varepsilon(\gamma)\\ \end{pmatrix}-\varepsilon(\alpha)bm\otimes\begin{pmatrix} a & 1_B\\ 1 & 1_A\\ \end{pmatrix}+bm\otimes\begin{pmatrix} a & 1_B\\ 1 & \varepsilon(\alpha)\\ \end{pmatrix}\\
-ma\varepsilon(\alpha)\otimes\begin{pmatrix} 1_A & 1_B\\ 1 & b\\ \end{pmatrix}+m\otimes\begin{pmatrix} a\varepsilon(\alpha) & 1_B\\ 1 & b\\ \end{pmatrix}.
\end{gathered}
\end{equation}
Next we see that
\begin{align*}
\Phi^2&\Big(\varepsilon(\beta)\varepsilon(\gamma)n\otimes1_A\otimes1_A-n\otimes\varepsilon(\beta)\otimes\varepsilon(\gamma)+n\varepsilon(\beta)\otimes1_A\otimes\varepsilon(\gamma)\\
&~-\varepsilon(\alpha)bm\otimes a\otimes1_A+bm\otimes a\otimes\varepsilon(\alpha)-ma\varepsilon(\alpha)\otimes1_A\otimes b+m\otimes a\varepsilon(\alpha)\otimes b\Big)\\
&=\varepsilon(\beta)\varepsilon(\gamma)n\otimes\begin{pmatrix} 1_A & 1_B\\ 1 & 1_A\\ \end{pmatrix}-n\otimes\begin{pmatrix} \varepsilon(\beta) & 1_B\\ 1 & \varepsilon(\gamma)\\ \end{pmatrix}+n\varepsilon(\beta)\otimes\begin{pmatrix} 1_A & 1_B\\ 1 & \varepsilon(\gamma)\\ \end{pmatrix}\\
&~-\varepsilon(\alpha)bm\otimes\begin{pmatrix} a & 1_B\\ 1 & 1_A\\ \end{pmatrix}+bm\otimes\begin{pmatrix} a & 1_B\\ 1 & \varepsilon(\alpha)\\ \end{pmatrix}-ma\varepsilon(\alpha)\otimes\begin{pmatrix} 1_A & 1_B\\ 1 & b\\ \end{pmatrix}+m\otimes\begin{pmatrix} a\varepsilon(\alpha) & 1_B\\ 1 & b\\ \end{pmatrix}\\
&=m\otimes\begin{pmatrix} a & \alpha\\ 1 & b\\ \end{pmatrix} \text{~~by \eqref{e8}}.
\end{align*}
Thus, we will have that $\Ker(\Psi)\subseteq\Ig(\Phi^2)$ if only we can show that
$$
\varepsilon(\beta)\varepsilon(\gamma)n\otimes1_A\otimes1_A-n\otimes\varepsilon(\beta)\otimes\varepsilon(\gamma)+n\varepsilon(\beta)\otimes1_A\otimes\varepsilon(\gamma)
$$
$$
-\varepsilon(\alpha)bm\otimes a\otimes1_A+bm\otimes a\otimes\varepsilon(\alpha)-ma\varepsilon(\alpha)\otimes1_A\otimes b+m\otimes a\varepsilon(\alpha)\otimes b
$$
is in $\Hg_2(A,M)$. For that, we need to show that it goes to zero under the map $d_2^A$.

Since $m\otimes\begin{pmatrix} a & \alpha\\ 1 & b\\ \end{pmatrix}\in\Hg_2((A,B,\varepsilon);M)$, we have that
\begin{align}\label{e9}
ma\varepsilon(\alpha)\otimes b-m\otimes a\varepsilon(\alpha)b+b\varepsilon(\alpha)m\otimes a=0.
\end{align}
Moreover, by applying $\partial_2^{\varepsilon}$ to both sides of \eqref{e1}, we get
\begin{equation*}
\begin{gathered}
bma\varepsilon(\alpha)\otimes1_A-bma\otimes\varepsilon(\alpha)+\varepsilon(\alpha)bma\otimes1_A\\
=n\varepsilon(\beta)\varepsilon(\gamma)\otimes1_A-n\varepsilon(\beta)\otimes\varepsilon(\gamma)+\varepsilon(\gamma)n\varepsilon(\beta)\otimes1_A-n\varepsilon(\beta)\varepsilon(\gamma)\otimes1_A\\
+n\otimes\varepsilon(\beta)\varepsilon(\gamma)-\varepsilon(\beta)\varepsilon(\gamma)n\otimes1_A+\varepsilon(\gamma)n\varepsilon(\beta)\otimes1_A-\varepsilon(\gamma)n\otimes\varepsilon(\beta)+\varepsilon(\beta)\varepsilon(\gamma)n\otimes1_A,
\end{gathered}
\end{equation*}
which simplifies to
\begin{equation}\label{e2}
\begin{gathered}
bma\varepsilon(\alpha)\otimes1_A-bma\otimes\varepsilon(\alpha)+\varepsilon(\alpha)bma\otimes1_A=\\
n\varepsilon(\beta)\varepsilon(\gamma)\otimes1_A-n\varepsilon(\beta)\otimes\varepsilon(\gamma)+n\otimes\varepsilon(\beta)\varepsilon(\gamma)-\varepsilon(\gamma)n\otimes\varepsilon(\beta)+\varepsilon(\beta)\varepsilon(\gamma)n\otimes1_A.
\end{gathered}
\end{equation}
Thus we have that
\begin{align*}
d_2^A\Big(&\varepsilon(\beta)\varepsilon(\gamma)n\otimes1_A\otimes1_A-n\otimes\varepsilon(\beta)\otimes\varepsilon(\gamma)+n\varepsilon(\beta)\otimes1_A\otimes\varepsilon(\gamma)\\
&~-\varepsilon(\alpha)bm\otimes a\otimes1_A+bm\otimes a\otimes\varepsilon(\alpha)-ma\varepsilon(\alpha)\otimes1_A\otimes b+m\otimes a\varepsilon(\alpha)\otimes b\Big)\\
&=\varepsilon(\beta)\varepsilon(\gamma)n\otimes1_A-\varepsilon(\beta)\varepsilon(\gamma)n\otimes1_A+\varepsilon(\beta)\varepsilon(\gamma)n\otimes1_A-n\varepsilon(\beta)\otimes\varepsilon(\gamma)\\
&~+n\otimes\varepsilon(\beta)\varepsilon(\gamma)-\varepsilon(\gamma)n\otimes\varepsilon(\beta)+n\varepsilon(\beta)\otimes\varepsilon(\gamma)-n\varepsilon(\beta)\otimes\varepsilon(\gamma)\\
&~+\varepsilon(\gamma)n\varepsilon(\beta)\otimes1_A-\varepsilon(\alpha)bma\otimes1_A+\varepsilon(\alpha)bm\otimes a-\varepsilon(\alpha)bm\otimes a\\
&~+bma\otimes\varepsilon(\alpha)-bm\otimes a\varepsilon(\alpha)+\varepsilon(\alpha)bm\otimes a-ma\varepsilon(\alpha)\otimes b\\
&~+ma\varepsilon(\alpha)\otimes b-bma\varepsilon(\alpha)\otimes1_A+ma\varepsilon(\alpha)\otimes b-m\otimes a\varepsilon(\alpha)b+bm\otimes a\varepsilon(\alpha)\\
&=n\varepsilon(\beta)\varepsilon(\gamma)\otimes1_A-n\varepsilon(\beta)\otimes\varepsilon(\gamma)+n\otimes\varepsilon(\beta)\varepsilon(\gamma)-\varepsilon(\gamma)n\otimes\varepsilon(\beta)\\
&~+\varepsilon(\beta)\varepsilon(\gamma)n\otimes1_A-bma\varepsilon(\alpha)\otimes1_A+bma\otimes\varepsilon(\alpha)-\varepsilon(\alpha)bma\otimes1_A\\
&~+ma\varepsilon(\alpha)\otimes b-m\otimes a\varepsilon(\alpha)b+\varepsilon(\alpha)bm\otimes a \text{~by simplifying}\\
&=bma\varepsilon(\alpha)\otimes1_A-bma\otimes\varepsilon(\alpha)+\varepsilon(\alpha)bma\otimes1_A-bma\varepsilon(\alpha)\otimes1_A+bma\otimes\varepsilon(\alpha)\\
&~-\varepsilon(\alpha)bma\otimes1_A+ma\varepsilon(\alpha)\otimes b-m\otimes a\varepsilon(\alpha)b+\varepsilon(\alpha)bm\otimes a \text{~by \eqref{e2}}\\
&=ma\varepsilon(\alpha)\otimes b-m\otimes a\varepsilon(\alpha)b+\varepsilon(\alpha)bm\otimes a \text{~by simplifying}\\
&=0 \text{~by \eqref{e9}},
\end{align*}
which was what we wanted. Hence $\Ker(\Psi)\subseteq\Ig(\Phi^2)$.

For $(iii)$, it suffices to show that $\varepsilon_*\circ\Psi=0$. We begin by taking $m\otimes\begin{pmatrix} a & \alpha\\ 1 & b\\ \end{pmatrix}\in\Hg_2((A,B,\varepsilon);M)$, and we want to conclude that $bma\otimes\varepsilon(\alpha)=0$ in $\Hg_1(A,M)$. Notice:
\begin{align}\label{a1}
ma\varepsilon(\alpha)\otimes b-m\otimes a\varepsilon(\alpha)b+ b\varepsilon(\alpha)m\otimes a=0
\end{align}
since $m\otimes\begin{pmatrix} a & \alpha\\ 1 & b\\ \end{pmatrix}\in\Hg_2((A,B,\varepsilon);M)$, as well as the two boundaries in $\Hg_1(A,M)$:
\begin{align}\label{a2}
d_2^A(bm\otimes a\otimes\varepsilon(\alpha))=bma\otimes\varepsilon(\alpha)-bm\otimes a\varepsilon(\alpha)+\varepsilon(\alpha)bm\otimes a,
\end{align}
and
\begin{align}\label{a3}
d_2^A(m\otimes a\varepsilon(\alpha)\otimes b)=ma\varepsilon(\alpha)\otimes b-m\otimes a\varepsilon(\alpha)b+bm\otimes a\varepsilon(\alpha).
\end{align}
Now we note that
\begin{align*}
bma\otimes\varepsilon(\alpha)&=bm\otimes a\varepsilon(\alpha)-\varepsilon(\alpha)bm\otimes a \text{~~by \eqref{a2}}\\
&=-ma\varepsilon(\alpha)\otimes b+m\otimes a\varepsilon(\alpha)b-\varepsilon(\alpha)bm\otimes a \text{~~by \eqref{a3}}\\
&=0 \text{~~by \eqref{a1}}.
\end{align*}
This establishes $(iii)$, and so $\Ig(\Psi)\subseteq\Ker(\varepsilon_*)$.

For $(iv)$, we take $m\otimes\alpha\in\Hg_1(B,M)$ such that $m\otimes\varepsilon(\alpha)=0$ in $\Hg_1(A,M)$. We want to show that $m\otimes\alpha$ is the image of some element under $\Psi$. Since $m\varepsilon(\alpha)\otimes1_A=0$ in $\Hg_1(A,M)$, we have that
$$m\varepsilon(\alpha)\otimes1_A-m\otimes\varepsilon(\alpha)+\varepsilon(\alpha)m\otimes1_A$$
also equals zero in $\Hg_1(A,M)$. Thus, this element is a boundary, which means there exists some $a,b\in A$ and $n\in M$ such that
$$d_2^A(n\otimes a\otimes b)=m\varepsilon(\alpha)\otimes1_A-m\otimes\varepsilon(\alpha)+\varepsilon(\alpha)m\otimes1_A.$$
Next note that
\begin{align*}
\partial_2^{\varepsilon}\left(m\otimes
\begin{pmatrix}
1_A & \alpha\\
1 & 1_A\\
\end{pmatrix}\right)&=m\varepsilon(\alpha)\otimes1_A-m\otimes\varepsilon(\alpha)+\varepsilon(\alpha)m\otimes1_A\\
&=d_2^A(n\otimes a\otimes b)\\
&=\partial_2^{\varepsilon}\left(n\otimes
\begin{pmatrix}
a & 1_B\\
1 & b\\
\end{pmatrix}\right).
\end{align*}
Since $\partial_2^{\varepsilon}\left(
m\otimes
\begin{pmatrix}
1_A & \alpha\\
1 & 1_A\\
\end{pmatrix}-n\otimes
\begin{pmatrix}
a & 1_B\\
1 & b\\
\end{pmatrix}
\right)=0$, we have $m\otimes
\begin{pmatrix}
1_A & \alpha\\
1 & 1_A\\
\end{pmatrix}-n\otimes
\begin{pmatrix}
a & 1_B\\
1 & b\\
\end{pmatrix}\in\Hg_2((A,B,\varepsilon);M)$. Finally notice that
$$\Psi\left(
m\otimes
\begin{pmatrix}
1_A & \alpha\\
1 & 1_A\\
\end{pmatrix}-n\otimes
\begin{pmatrix}
a & 1_B\\
1 & b\\
\end{pmatrix}
\right)=m\otimes\alpha-bna\otimes1_B=m\otimes\alpha.$$
Hence $\Ker(\varepsilon_*)\subseteq\Ig(\Psi)$.

For $(vi)$, we take $m\otimes a\in\Hg_1(A,M)$ such that $m\otimes a=0$ in $\Hg_1((A,B,\varepsilon);M)$. We want to show that $m\otimes a$ is the image of some element under $\varepsilon_*$. Since $m\otimes a=0$ in $\Hg_1((A,B,\varepsilon);M)$, this means that it is a boundary. Therefore, there exists some $b,c\in A$, $n\in M$, and $\alpha\in B$ such that
$$\partial_2^{\varepsilon}\left(n\otimes
\begin{pmatrix} b & \alpha\\ 1 & c\\ \end{pmatrix}\right)=m\otimes a.$$
Observe:
\begin{align}\label{b1}
nb\varepsilon(\alpha)\otimes c-n\otimes b\varepsilon(\alpha)c+c\varepsilon(\alpha)n\otimes b=m\otimes a
\end{align}
by above, as well as the two boundaries in $\Hg_1(A,M)$:
\begin{align}\label{b2}
d_2^A(cn\otimes b\otimes\varepsilon(\alpha))=cnb\otimes\varepsilon(\alpha)-cn\otimes b\varepsilon(\alpha)+\varepsilon(\alpha)cn\otimes b,
\end{align}
and
\begin{align}\label{b3}
d_2^A(n\otimes b\varepsilon(\alpha)\otimes c)=nb\varepsilon(\alpha)\otimes c-n\otimes b\varepsilon(\alpha)c+cn\otimes b\varepsilon(\alpha).
\end{align}
Now we note that
\begin{align*}
-cnb\otimes\varepsilon(\alpha)&=-cn\otimes b\varepsilon(\alpha)+\varepsilon(\alpha)cn\otimes b \text{~~by \eqref{b2}}\\
&=nb\varepsilon(\alpha)\otimes c-n\otimes b\varepsilon(\alpha)c+\varepsilon(\alpha)cn\otimes b \text{~~by \eqref{b3}}\\
&=m\otimes a \text{~~by \eqref{b1}}.
\end{align*}
Thus, we notice that $-cnb\otimes\alpha\in\Hg_1(B,M)$ because $-cnb\varepsilon(\alpha)+\varepsilon(\alpha)cnb=0$ due to the fact that $M$ is $B$-symmetric, and
$$\varepsilon_*(-cnb\otimes\alpha)=-cnb\otimes\varepsilon(\alpha)=m\otimes a.$$
This establishes $\Ker(\Phi^1)\subseteq\Ig(\varepsilon_*)$ and completes our proof.
\end{proof}

\begin{corollary}[First Fundamental Exact Sequence for $\Omega$]\emph{(\cite{Mat},\cite{W})}
Let $k\longrightarrow B\longrightarrow A$ be morphisms of commutative algebras. Then there is an exact sequence of $A$-modules:
$$
A\otimes_B\Omega_{B|k}^1\longrightarrow\Omega_{A|k}^1\longrightarrow\Omega_{A|B}^1\longrightarrow0.
$$
\end{corollary}
\begin{proof}
Notice that we have the morphisms $k\longrightarrow B\longrightarrow A$, the first coming from $B$ being a $k$-algebra, and the second being $\varepsilon$. Apply Propositions \ref{HH1A}, \ref{Kahler2}, and \ref{result} with $A$ commutative and $M=A$.
\end{proof}

\begin{example}
Since $\Hg_n(k,M)=0$ for all $n>0$, note that $\Hg_1((k,B,\varepsilon);M)=0$ and $\Hg_2((k,B,\varepsilon);M)\cong\Hg_1(B,M)\cong M\otimes_B\Omega_{B|k}^1$ as consequence of Propositions \ref{HH1A} and \ref{result}. Again using the exact sequence \eqref{sequence}, one has $\Hg_1((A,A,\id);M)=\Hg_2((A,A,\id);M)=0$.
\end{example}

\subsection{Functoriality}

Recall that for the usual Hochschild homology, $\Hg_\bullet(A,M)$ is a covariant functor in $M$. It can also be seen as functorial in $A$ in a certain sense. In this section we establish similar results for the secondary Hochschild homology.

First we introduce the category of triples $(A,B,\varepsilon)$ over $k$, denoted $\T-k$. Here the objects are triples $(A,B,\varepsilon)$, and a morphism between two triples $(A,B,\varepsilon)$ and $(A',B',\varepsilon')$ is a pair $(f,g)$ where $f:A\longrightarrow A'$ and $g:B\longrightarrow B'$ are morphisms of $k$-algebras such that $f\circ\varepsilon=\varepsilon'\circ g$. In other words, the following diagram commutes:
\begin{equation}\label{ComD}
\begin{aligned}
\begin{tikzpicture}[scale=3]
\node (a) at (0,0) {$B'$};
\node (b) at (1.5,0) {$A'$};
\node (c) at (0,.75) {$B$};
\node (d) at (1.5,.75) {$A$};
\path[->,font=\small,>=angle 90]
(a) edge node [above] {$\varepsilon'$} (b)
(d) edge node [right] {$f$} (b)
(c) edge node [left] {$g$} (a)
(c) edge node [above] {$\varepsilon$} (d);
\end{tikzpicture}
\end{aligned}
\end{equation}
Composition is done in the natural way, and it is easy to verify that $\T-k$ is a category.

\begin{remark}
Secondary Hochschild homology is functorial in $M$ since $f:M\longrightarrow M'$ induces a map
$$f_*:\Hg_\bullet((A,B,\varepsilon);M)\longrightarrow\Hg_\bullet((A,B,\varepsilon);M')$$
where
$$f_*\left(m
\otimes
\begin{pmatrix}
a_1 & b_{1,2} & \cdots & b_{1,n-1} & b_{1,n}\\
1 & a_2 & \cdots & b_{2,n-1} & b_{2,n}\\
\vdots & \vdots & \ddots & \vdots & \vdots\\
1 & 1 & \cdots & a_{n-1} & b_{n-1,n}\\
1 & 1 & \cdots & 1 & a_n\\
\end{pmatrix}
\right)=
f(m)\otimes\begin{pmatrix}
a_1 & b_{1,2} & \cdots & b_{1,n-1} & b_{1,n}\\
1 & a_2 & \cdots & b_{2,n-1} & b_{2,n}\\
\vdots & \vdots & \ddots & \vdots & \vdots\\
1 & 1 & \cdots & a_{n-1} & b_{n-1,n}\\
1 & 1 & \cdots & 1 & a_n\\
\end{pmatrix}.$$
Secondary Hochschild homology is also functorial in $(A,B,\varepsilon)$ in a certain way. Let the pair $(f,g):(A,B,\varepsilon)\longrightarrow(A',B',\varepsilon')$ be a morphism of triples. Furthermore, let $M'$ be an $A'$-bimodule which is $B'$-symmetric. Notice that $M'$ can be considered an $A$-bimodule under the rule
$$a\cdot m'=f(a)m' \hspace{.15in}\text{and}\hspace{.15in} m'\cdot a=m'f(a).$$
It can also be considered $B$-symmetric by using \eqref{ComD} because
$$\alpha\cdot m'=f(\varepsilon(\alpha))m'=\varepsilon'(g(\alpha))m'=m'\varepsilon'(g(\alpha))=m'f(\varepsilon(\alpha))=m'\cdot\alpha.$$
Thus $(f,g)$ induces a map
$$(f,g)_*:\Hg_\bullet((A,B,\varepsilon);M')\longrightarrow\Hg_\bullet((A',B',\varepsilon');M')$$
where
$$
(f,g)_*\left(m'
\otimes
\begin{pmatrix}
a_1 & b_{1,2} & \cdots & b_{1,n}\\
1 & a_2 & \cdots & b_{2,n}\\
\vdots & \vdots & \ddots & \vdots\\
1 & 1 & \cdots & a_n\\
\end{pmatrix}
\right)
=m'\otimes\begin{pmatrix}
f(a_1) & g(b_{1,2}) & \cdots & g(b_{1,n})\\
1 & f(a_2) & \cdots & g(b_{2,n})\\
\vdots & \vdots & \ddots & \vdots\\
1 & 1 & \cdots & f(a_n)\\
\end{pmatrix}.
$$
\end{remark}

\begin{remark}
Notice that when one takes $B=k$, this reduces to the usual case where the Hochschild homology is functorial in $M$ and can be viewed as functorial in $A$.
\end{remark}

\section*{Acknowledgement}

I would like to thank my advisor Mihai D. Staic for some suggestions towards the preparation of this document.


\end{document}